\documentclass[11pt]{amsart}

\usepackage{amsmath,amsthm,color}

\usepackage{amssymb}

\usepackage{pdfsync}

\newcommand{\R}{{\mathbb R}}       
\newcommand{\B}{{\mathbb B}}

\newcommand{\HH}{{\mathcal H}}

\newcommand{\cH}{{\mathcal{H}}}

\newcommand{\diam}{{\rm diam}}

\newcommand{\fiproof}{{\hspace*{\fill} $\square$ \vspace{2pt}}}

\newcommand{\rf}[1]{{\eqref{#1}}}
\newcommand{\supp}{\operatorname{supp}}

\newcommand{\vphi}{{\varphi}}
\newcommand{\ve}{{\varepsilon}}
\newcommand{\vv}{{\vspace{2mm}}}
\newcommand{\vvv}{{\vspace{3mm}}}
\newcommand{\wt}[1]{{\widetilde{#1}}}

\newtheorem{theorem}{Theorem}[section]
\newtheorem*{theorem*}{Theorem}
\newtheorem*{lemma*}{Lemma}
\newtheorem*{theorema*}{Theorem A}
\newtheorem*{theoremb*}{Theorem B}
\newtheorem*{theoremc*}{Theorem C}
\newtheorem*{mainlemma*}{Main Lemma}

\newtheorem{lemma}[theorem]{Lemma}

\theoremstyle{definition}

\def\XXint#1#2#3{{\setbox0=\hbox{$#1{#2#3}{\int}$ }
\vcenter{\hbox{$#2#3$ }}\kern-.58\wd0}}

\theoremstyle{remark}

\numberwithin{equation}{section}

\begin{document}

\title{On Tsirelson's theorem about triple points for harmonic measure
}

\author[Tolsa]{Xavier Tolsa}
\address{Xavier Tolsa
\\
ICREA, Passeig Llu\'\i ­s Companys 23 08010 Barcelona, Catalonia, and\\
Departament de Matem\`atiques and BGSMath
\\
Universitat Aut\`onoma de Barcelona
\\
Edifici C Facultat de Ci\`encies
\\
08193 Bellaterra (Barcelona), Catalonia
}
\email{xtolsa@mat.uab.cat}

\thanks{X.T. was supported by the ERC grant 320501 of the European Research Council (FP7/2007-2013), and also partially supported by 2014-SGR-75 (Catalonia), MTM2013-44304-P (Spain), and by the Marie Curie ITN MAnET (FP7-607647). A.V. was partially supported by the NSF grant DMS-160065.
}

\author{Alexander Volberg}

\address{Alexander Volberg
\\
Department of Mathematics
\\
Michigan State University
\\
East Lan\-sing, MI 48824, USA}

\email{volberg@math.msu.edu}

\begin{abstract}
A theorem of Tsirelson from 1997 asserts that given three disjoint domains
in $\R^{n+1}$, the set of triple points belonging to the intersection of the three boundaries where the three corresponding harmonic measures are mutually absolutely continuous has null harmonic measure. The original proof by Tsirelson is based on the fine analysis of filtrations for Brownian and Walsh-Brownian motions and can not
be translated into potential theory arguments. 
In the present paper we give a purely analytical proof of the same result.

\end{abstract}

\maketitle

\section{Introduction}

In a paper from 1997 Tsirelson \cite{Tsirelson}
proved the following result, previously conjectured by Bishop in \cite{Bishop-arkiv} (see also Problem {\bf a} in \cite{EFS2}):

\begin{theorem}\cite{Tsirelson}\label{teo}
Let $\Omega_1,\Omega_2,\Omega_3\subset\R^{n+1}$ be disjoint open connected sets,
with harmonic measures $\omega_1,\omega_2,\omega_3$.
Let $E\subset \partial\Omega_1\cap\partial\Omega_2\cap\partial\Omega_3$ so that
$\omega_1,\omega_2,\omega_3$ are mutually absolutely continuous in $E$.
Then $\omega_i(E)=0$ for $i=1,2,3$.
\end{theorem}

We remark that the planar case $n=1$ of the preceding theorem had previously
been proved by Bishop \cite{Bishop-arkiv} and Eremenko, Fuglede, and Sodin 
\cite{EFS1}. In the higher dimensional case, another previous partial result had been obtained by Bishop
in \cite{Bishop-arkiv}. Namely he had shown that if $\Omega_1,\ldots,\Omega_m\subset\R^{n+1}$ are disjoint domains
with harmonic measures $\omega_1,\ldots,\omega_m$ which are mutually absolutely continuous in $E\subset 
\bigcap_{j=1}^m\partial\Omega_j$, then $\omega_j(E)=0$ if $m=5$ in $\R^3$, or if $m=11$ in any dimension.

The original proof of Theorem \ref{teo} by Tsirelson in \cite{Tsirelson} is mostly based on
the fine analysis of filtrations for Brownian and Walsh-Brownian motions, and
so this proof can not
be directly translated into a potential theory proof.
In \cite{Tsirelson} Tsirelson asked if the same result can be achieved
by non-stochastic arguments. In the present note  we give a quite short  proof of this theorem by using purely analytical arguments.

It is also worth pointing out that the original formulation of Tsirelson's
theorem in \cite{Tsirelson} is somewhat different from the one above. Indeed, denote by
$\omega_1\wedge\omega_2\wedge\omega_3$ the greatest measure $\mu$ such that
$\mu\leq\omega_i$ for $i=1,2,3$. It is shown in \cite{Tsirelson}
that if $\Omega_1,\Omega_2,\Omega_3\subset\R^{n+1}$ are disjoint open connected sets with harmonic measures $\omega_1,\omega_2,\omega_3$, then
$\omega_1\wedge\omega_2\wedge\omega_3=0$. It is immediate to check that
that this statement is equivalent to the one in Theorem \ref{teo}. In fact, assume that Theorem \ref{teo} holds and let $\mu=\omega_1\wedge\omega_2\wedge\omega_3$. By the Radon-Nikodym theorem there are non-negative functions 
$g_i$ such that $\mu=g_i\,\omega_i$. Letting $E_i=\supp\omega_i\cap\{g_i>0\}$, it follows
that 
$$\mu = \mu|_{E_i} \approx \omega_i|_{E_i} \quad \mbox{for $i=1,2,3,$}$$
where ``$\approx$" denotes mutual absolute continuity. It follows then
that 
$$\mu=\mu|_{E_1\cap E_2\cap E_3} \approx \omega_i|_{E_1\cap E_2\cap E_3}
\quad \mbox{for $i=1,2,3,$}.$$
Now, by Theorem \ref{teo} the mutual absolute continuity of $\omega_1$,
$\omega_2$ and $\omega_3$ in $E_1\cap E_2\cap E_3$ implies that 
$\omega_i(E_1\cap E_2\cap E_3) =0$ and thus $\mu\equiv0$. The converse implication follows by analogous arguments.

Tsirelson's theorem is connected to a
 recent result
by Azzam, Mourgoglou and the first author of this paper. To state this, we need to introduce
the capacity density condition (CDC). A domain $\Omega\subset\R^{n+1}$, $n\geq2$, satisfies the CDC if
there is $R_{\Omega}>0$ and $c_\Omega>0$ so that ${\rm Cap}(B\setminus \Omega)\geq c_\Omega \,r(B)^{n-1}$ for any ball $B$ centered on $\partial\Omega$ of radius $r(B)\in (0,R_{\Omega})$, where ${\rm Cap}$ stands for the Newtonian capacity.
 
\begin{theorem}\label{teo*}\cite{AMT}
For $n\geq 2$, 
 let $\Omega_1\subset \R^{n+1}$ be  open and let $\Omega_2= \bigl(\,\overline{\Omega_1}\,\bigr)^c$. Assume that $\Omega_1,\Omega_2$ are both connected and satisfy the capacity density condition and $\partial\Omega_1 = \partial\Omega_2$.
Let $\omega_1,\omega_2$ be the respective harmonic measures of $\Omega_1,\Omega_2$.
Let $E\subset \partial\Omega_1$ be a Borel set and let $T$ the set of tangent points for $\partial\Omega_1$. Then $\omega_1\perp \omega_2$ on $E$ if and only if $\cH^{n}(E\cap T)=0$. Further, if $\omega_1\ll\omega_2\ll\omega_1$ on $E$, then $E$ contains an $n$-rectifiable subset $F$ upon which $\omega_1$ and
$\omega_2$ are mutually absolutely continuous with respect to $\cH^{n}$.
\end{theorem}

It would be natural to think that perhaps Tsirelson's theorem may be derived as a corollary of Theorem \ref{teo*} in the case $n\geq2$. 
  As far as we know, this is not the case. The main reason is that the latter theorem requires the CDC. In fact, this 
  assumption, which is not present in
Theorem \ref{teo}, seems to be  an essential
condition for the blowup techniques\footnote{Actually, after the present paper was finished, in a joint of the authors with Azzam
and Mourgoglou \cite{AMTV} it was shown that Theorem \ref{teo*} also holds without the CDC condition.
The main novelty with respect to the arguments in \cite{AMT} is the use of a blowup argument which
is precisely inspired by the techniques in the current paper.}
 used in \cite{AMT}. So our
 arguments below to prove Theorem \ref{teo} are independent of the ones in \cite{AMT}. On the other hand, to tell the truth, we will also use a blowup argument to prove Tsirelson's 
 theorem, which has some similarities (and also some big differences) with other previous blowup arguments, such as the ones by Kenig, Preiss, and Toro in \cite{KPT}, or the ones in \cite{AMT}. See the 
 next section for more details in this direction.
\vvv

Acknowledgement: We are grateful to Misha Sodin for calling our attention to Tsirelson's theorem and the lack of an analytic proof of this in a conference in Bedlewo (Poland) in August 2016.

\vvv

\section{Preliminary discussion}

As usual in harmonic analysis, we denote by $C$ or $c$ constants which usually only depend on the dimension $n$ and
other fixed parameters, and which may change their values at different occurrences. For $a,b\geq 0$, we will write $a\lesssim b$ if there is $C>0$ so that $a\leq Cb$ and $a\lesssim_{t} b$ if the constant $C$ depends on the parameter $t$. We write $a\approx b$ to mean $a\lesssim b\lesssim a$ and define $a\approx_{t}b$ similarly.

In this paper we assume that the harmonic measure any domain of is constructed by Perron's method.

We will prove Tsirelson's theorem by applying a blowup argument, as in other
previous
works such as \cite{KPT} or \cite{AMT}. 
As far as we know, the introduction of blowup arguments (also called renormalisation arguments) in the study of harmonic
measure is basically due to Kenig and Toro (see \cite{Kenig-Toro-annals}). These techniques are frequent in other fields such as calculus of variations, sometimes in combination with monotonicity formulas. See \cite{Caffarelli-Salsa} for a standard reference in this area.
The arguments in \cite{KPT}
assume that the domains involved are non-tangentiallly accessible (NTA), while in \cite{AMT}
the domains satisfy the CDC. In both cases these assumptions imply that the associated Green functions are H\"older
continuous, which is an essential ingredient in the blowup arguments of both works.

So a first idea when trying to implement a blowup argument to prove
Theorem \ref{teo}
might consist in showing that at least two 
of the Green functions of the domains $\Omega_1$, $\Omega_2$, and $\Omega_3$ are H\"older continuous.
To this end, let $E$ be as in Theorem \ref{teo} and let $x\in E$.
 It is clear that for each radius $r>0$ there exists at least
two indices $1\leq i <j\leq 3$ such that 
$$\HH^{n+1}(\Omega_i^c\cap B(x,r))\gtrsim r^{n+1}
\quad \text{and}\quad \HH^{n+1}(\Omega_j^c\cap B(x,r))\gtrsim r^{n+1}.$$
Here $\HH^{n+1}$ stands for the $(n+1)$-Hausdorff measure.
So given $r_0$ and $k_0\geq1$, it follows that there are two indices $1\leq i <j\leq 3$ such that 
$$\HH^{n+1}(\Omega_i^c\cap B(x,2^{-k}r))\gtrsim (2^{-k}r)^{n+1}
\quad \text{and}\quad \HH^{n+1}(\Omega_j^c\cap B(x,2^{-k}r))\gtrsim (2^{-k}r)^{n+1}$$
for at least $k_0/3$ integers $k\in[1,k_0]$. 

From the preceding discussion, by standard arguments
analogous to the ones when the CDC holds (arguing as in Lemma 4.5 of \cite{AMT}, say), we get that for each $x\in E$ 
and $0<r\leq r_0$, there are at least two indices $1\leq i <j\leq 3$ such 
that
$$G_i(p_i,y) \leq C \sup_{z\in B(x,Cr_0)} G_i(p_i,z)\,\left(\frac r{r_0}\right)^\alpha \quad \mbox{
for $y\in B(x,r)$,}$$
for some $\alpha>0$,
and the same replacing $i$ by $j$. Here $G_h(p_h,\cdot)$ stands for the
Green function of $\Omega_h$ with pole $p_h\in\Omega_h$, which we assume
to be deep inside $\Omega_h$.
Note that the precise indices $i,j$ above depend on the particular point $x$, and more important, also on $r$. As far as we know, from the preceding estimate we cannot infer that for each point $x\in E$ there are two indices $h=i,j$
(or even one index $h$) such that $G_h(p_h,\cdot)$ is H\"older continuous at $x$. In fact, by compactness it seems that at most we will get two indices $i,j$
and a sequence of radii $r_k\to 0$ (depending on $x$) so that
\begin{equation}\label{eq1}
G_h(p_h,y) \leq C \sup_{z\in B(x,Cr_0)} G_h(p_h,z)\,\left(\frac {r_k}{r_0}\right)^\alpha \quad \mbox{
for $y\in B(x,r_k)$.}
\end{equation}
Unfortunately, this condition is much weaker than H\"older continuity at $x$ for
$G_i(p_i,\cdot)$ and $G_j(p_j,\cdot)$.

We have not been able to use the condition \rf{eq1} to extend the arguments
in \cite{KPT} or \cite{AMT} to our particular context. Instead, our 
arguments will rely on the strong convergence in $L^2$ of suitable sequences of 
Green functions, which can be derived by applying the Rellich-Kondrachov theorem
in combination with Caccioppoli's inequality, for example.

\vvv

\section{Proof of Tsirelson's theorem}

\subsection{Preliminary lemmas}

We will need the following classical result. See for example \cite[Lemma 3.3]{AHM3TV} for the detailed proof.

\begin{lemma}\label{l:w>G}
Let $n\ge 2$ and $\Omega\subset\R^{n+1}$ be a bounded domain. Denote by $\omega^p$ its harmonic measure with pole at $p$ and by $G$ its Green function.
Let $B=B(x_0,r)$ be a closed ball with $x_0\in\partial \Omega$ and $0<r<\diam(\partial \Omega)$. Then, for all $a>0$,
\begin{equation}\label{eq:Green-lowerbound}
 \omega^{x}(aB)\gtrsim \inf_{z\in 2B\cap \Omega} \omega^{z}(aB)\, r^{n-1}\, G(x,y)\quad\mbox{
 for all $x\in \Omega\backslash  2B$ and $y\in B\cap\Omega$,}
 \end{equation}
 with the implicit constant independent of $a$.
\end{lemma}

In the preceding statement, $aB$ stands for the ball concentric with $B$ with radius equal to $a$ times the radius of $B$.

The next lemma is usually known as Bourgain's estimate. See \cite{Bourgain}
(or \cite[Lemma 3.4]{AHM3TV} for the precise formulation below).

\begin{lemma}
\label{lembourgain}
There is $\delta_{0}>0$ depending only on $n\geq 1$ so that the following holds for $\delta\in (0,\delta_{0})$. Let $\Omega\subset \R^{n+1}$ be a  bounded domain, $n-1<s\le n+1$,  $\xi \in \partial \Omega$, $r>0$, and $B=B(\xi,r)$. Then 
\[ \omega^{x}(B)\gtrsim_{n,s} \frac{\mathcal H_\infty^{s}(\partial\Omega\cap \delta B)}{(\delta r)^{s}}\quad \mbox{  for all }x\in \delta B\cap \Omega .\]
\end{lemma}

Let $\xi\in E$ and $r>0$. 
Suppose that 
\begin{equation}\label{eq300}
\HH^{n+1}(B(\xi,r)\cap \Omega_3) = \max_{1\leq i \leq3}\HH^{n+1}(B(\xi,r)\cap \Omega_i).
\end{equation}
Then
\begin{equation}\label{eq300a}
\HH^{n+1}(B(\xi,r)\cap \Omega_1) + \HH^{n+1}(B(\xi,r)\cap \Omega_2) \leq \frac23 \,\HH^{n+1}(B(\xi,r)).
\end{equation}
From Lemmas \ref{l:w>G} and \ref{lembourgain} we deduce that if \rf{eq300} holds for $\xi\in E$ and $r>0$, then, for $i=1,2$, 
\begin{equation}\label{eqkey300}
\omega_i^x(B(\xi,\delta_0^{-1}r)\gtrsim r^{n-1}\,G_i(x,y)\quad
\mbox{for all $x\in  \Omega_i\setminus B(\xi,2r)$ and all $y\in B(\xi,r)\cap \Omega$.}
\end{equation}

The next lemma follows by standard arguments. For completeness, we will
give full details.

\begin{lemma}\label{lemdob}
Let $a>1$. Let $\omega_1,\omega_2,\omega_3$ and $E$ be as in Theorem \ref{teo}.
There exists $b=b(a)>1$ such that for $\omega_i$-a.e.\ $\xi\in E$ there exists a sequence of
$\omega^j$-$(a,b)$-doubling balls $B(\xi,r_k)$ simultaneously for $j=1,2,3$, with $r_k\to0$. That is,
$$\omega_j(B(\xi,ar_k))\leq b\,\omega_j(B(\xi,r_k))\qquad \mbox{for $j=1,2,3$ and for all $k\geq1$.}$$ 
\end{lemma}

\begin{proof}
It is well know that, given any Radon measure $\mu$ in $\R^{n+1}$,
 if we choose $b>a^{n+1}>1$, then for $\mu$-a.e. $\xi\in\R^{n+1}$ there
 exists a sequence of $\mu$-$(a,b)$-doubling balls $B(\xi,r_k)$, with 
 $r_k\to0$. See for example Chapter 2 of \cite{Tolsa-llibre}. Applying this to $\omega_1$, we infer that for $\omega_1$-a.e.
 $\xi\in E$ there
 exists a sequence of $\omega_1$-$(a,b)$-doubling balls $B(\xi,r_k)$ with 
 $r_k\to0$.

We claim now for $\omega_1$-a.e.\ $\xi\in E$ and $j=2,3$,
\begin{equation}\label{claim400}
\lim_{r\to0} \frac{\omega_1(B(\xi,ar))}{\omega_1(B(\xi,r))}\,\frac{\omega_j(B(\xi,r))}{\omega_j(B(\xi,ar))} = 1.
\end{equation}
To check this, note that by the Lebesgue differentiation theorem,
taking also into account the mutual absolute continuity of $\omega_j$ with
$\omega_1$ in $E$,
$$\lim_{r\to0} \frac{\omega_1(B(\xi,ar))}{\omega_1(B(\xi,r))}\,\frac{\omega_j(B(\xi,r))}{\omega_j(B(\xi,ar))} =
\lim_{r\to0} \frac{\omega_1(B(\xi,ar)\cap E)}{\omega_1(B(\xi,r)\cap E)}\,\frac{\omega_j(B(\xi,r)\cap E)}{\omega_j(B(\xi,ar)\cap E)}
\quad \mbox{ for $\omega_1$-a.e. $\xi\in E$.}$$
Let $h_{j,1}$ be the density functions such that
$$\omega_j|_E = h_{j,1}\,\omega_1|_E,$$
so that  $h_{j,1}\in L^1(\omega_1|_E)$ and $0<h_{j,1}<\infty$ $\omega_1$-a.e.
By the Lebesgue differentiation theorem again, we have
$$
\lim_{r\to0} \frac{\omega_1(B(\xi,ar)\cap E)}{\omega_j(B(\xi,ar)\cap E)}
= \frac1{h_{j,1}(\xi)}
\quad \mbox{ and }\quad
\lim_{r\to0} \frac{\omega_j(B(\xi,r)\cap E)}
{\omega_1(B(\xi,r)\cap E)}= h_{j,1}(\xi)$$
for a.e. $\xi\in E$, and thus our claim \rf{claim400} holds.

Now we deduce that if $\xi\in E$ is a point such that \rf{claim400} holds
and $B(\xi,r_k)$ is a sequence of $\omega_1$-$(a,b)$-doubling balls with
$r_k\to0$, then
$$\limsup_{k\to\infty} \frac{\omega_j(B(\xi,ar_k))}{\omega_j(B(\xi,r_k))}
= \limsup_{k\to\infty} \frac{\omega_1(B(\xi,ar_k))}{\omega_1(B(\xi,r_k))} \leq b,$$
which proves the lemma.
\end{proof}

\subsection{Proof of the theorem} \label{subdd**}
We will give now the detailed arguments for the case $n\geq2$, and later we will sketch the required changes for the planar case $n=1$.

By standard arguments we may assume the domains $\Omega_i$ to be bounded.
We fix poles $p_i\in\Omega_i$ for the harmonic measures $\omega_i$, with $p_i$ deep inside $\Omega_i$, and for simplicity we write $\omega_i =\omega^{p_i}_i$.
We denote by $h_{i,j}$ the density function of $\omega_i$ with respect to $\omega_j$ on $E$. That is,
$$\omega_i|_E = h_{i,j}\,\omega_j|_E.$$
Let $\xi\in E$ be a Lebesgue point for $\chi_E$ and for all the density functions $h_{i,j}$ and so that there exists a decreasing
sequence of radii $r_k\to0$ satisfying the property described in Lemma \ref{lemdob}, for some constant
$a>2$ big enough to be chosen below.
We may assume that there exists an infinite subsequence of radii
such that
$$\HH^{n+1}(B(\xi,r_k)\cap \Omega_3) = \max_{1\leq i \leq3}\HH^{n+1}(B(\xi,r_k)\cap \Omega_i).$$
By renaming the subsequence $\{r_k\}_k$ if necessary, we assume that this holds for all $k\geq1$.

Denote $\B=B(0,1)$ and consider the affine map $T_k(x) = (x-\xi)/r_k$, so that $T_k(B(\xi,r_k))=\B$.
For $i=1,2,3$ and $k\geq1$ take the measures
$$\omega_i^k = \frac1{\omega_i(B(\xi,r_k))}\,T_k\#\omega_i.$$
Notice that
$$1=\omega_i^k(\B)\leq \omega_i^k(a\B) = \frac{\omega_i(B(\xi,ar_k))}{\omega_i(B(\xi,r_k))}\leq b,$$
where $a\B=B(0,a)$.
Hence there is a subsequence of radii $r_k$ so that
$\omega_i^k\to\omega_i^\infty$ weakly in $\frac a2\B$ as $k\to\infty$, for some Borel measure $\omega_i^\infty$ such that
$$1\leq \omega_i^\infty(\overline \B)\leq\omega_i^\infty(\tfrac a2 \B)\leq b$$
for $i=1,2,3$ and all $k\geq1$.

For $i=1,2,3$ and $k\geq 1$ consider now the functions
\begin{equation}\label{330}
u_i^k(x) = \frac{r_k^{n-1}}{\omega_i(B(\xi,r_k))}\,G_i(p_i,T_k^{-1}(x)),
\end{equation}
so that, for any $C^\infty$ compactly supported function $\vphi$, we have
\begin{align}\label{eq301}
\int \vphi\,d\omega_i^k &= \frac{1}{\omega_i(B(\xi,r_k))}\int \vphi\circ T_k\,d\omega_i =
\frac{1}{\omega_i(B(\xi,r_k))}\int \Delta(\vphi\circ T_k)\,G_i(p_i,x)\,dx \\
& = 
\frac{1}{r_k^2\,\omega_i(B(\xi,r_k))}\int \Delta\vphi(T_kx)\,G_i(p_i,x)\,dx = 
\frac{r_k^{n-1}}{\omega_i(B(\xi,r_k))}\int \Delta\vphi(y)\,G_i(p_i,T_k^{-1}y)\,dy\nonumber\\
&= \int \Delta\vphi\,u_i^k\,dy.\nonumber
\end{align}
Notice also that $u_i^k$ is a non-negative function which is harmonic in 
$a\B\cap T_k(\Omega_i)$. Further, for $i=1,2$, by \rf{eqkey300},
assuming $r_k$ small enough and choosing $a>\delta_0^{-1}$, for all $x\in \delta_0a\B\cap T_k(\Omega_i)$,
\begin{equation}\label{eq331}
u_i^k(x)\leq C(b).
\end{equation}

We suppose that all the points in the open sets $\Omega_i$ are Wiener regular for the Dirichlet problem (otherwise we may apply an approximation argument analogous the one in \cite{HMMTV}. See the end of this section for more details). Then the functions
$u_i^k$ extend continuously by zero in $a\B\setminus T_k(\Omega_i)$. We continue to denote by $u_i^k$ such extensions, which are subharmonic in $a\B$.
By Caccioppoli's inequality and the uniform boundedness of $u_i^k$ in $\delta_0a\B$ we deduce that, for $i=1,2$,
$$\|\nabla u_i^k\|_{L^{2}(\tfrac14 \delta_0a\B)}\lesssim \|u_i^k\|_{L^2(\tfrac12 \delta_0a\B)}\lesssim_b 1.$$
See (3.7) of \cite{KPT} for a similar argument.
By the Rellich-Kondrachov theorem, the unit ball of the Sobolev space $W^{1,2}(\tfrac14 \delta_0a\B)$ is relatively
compact in $L^2(\tfrac14 \delta_0a\B)$, and thus there exists a subsequence of the functions $u_i^k$ which
converges {\em strongly} in $L^2(\tfrac14 \delta_0a\B)$ to another function $u_i\in L^2(\tfrac14 \delta_0a\B)$.
Passing to a subsequence, we assume that the whole sequence of functions $u_i^k$ converges 
in $L^2(\tfrac14 \delta_0a\B)$ to $u_i\in L^2(\tfrac14 \delta_0a\B)$.
In particular, from \rf{eq301}, passing to the limit  if follows that
\begin{equation}\label{eq302}
\int \vphi\,d\omega_i^\infty = \int \Delta\vphi\,u_i\,dx,
\end{equation}
for any $C^\infty$ function $\vphi$ compactly supported in $\tfrac14\delta_0a\B$.

Consider the function
$$u(x)= u_1(x) - u_2(x).$$
Note that, by \rf{eq302},
$$\int \Delta\vphi\,u\,dx = \int \Delta\vphi\,u_1\,dx - \int \Delta\vphi\,u_2\,dx =
\int \vphi\,d\omega_1^\infty - \int \vphi\,d\omega_2^\infty.$$
We claim that
\begin{equation}\label{eq310}
\omega_1^\infty =\omega_2^\infty\quad  \mbox{ in $\frac12a\B$.}
\end{equation}
We defer the detailed (and standard) arguments to the end of the proof.
Assuming \rf{eq310} for the moment, we deduce that $\Delta u=0$ in $\tfrac14\delta_0a\B$ in the sense of distributions, and hence also in the classical sense
(because $u\in L^2(\tfrac14\delta_0a\B)$).

Let us check that $u$ does not vanish identically in $\tfrac14\delta_0a\B$.
Since the domains $\Omega_1$ and $\Omega_2$
are disjoint, it follows that
$$\|u_1^k - u_2^k\|_{L^2(\tfrac14\delta_0a\B)} \geq 
\|u_1^k\|_{L^2(\tfrac14\delta_0a\B)} \gtrsim \|u_1^k\|_{L^1(\tfrac14\delta_0a\B)},$$
which is bounded from below by \rf{eq302}. To see this, just take a bump function $\vphi$ identically $1$ on $\overline \B$ and supported on $2\B\subset \tfrac14\delta_0a\B$ (assume $a\geq 4\delta_0^{-1}$).
Hence, by the convergence of $u_1^k -  u_2^k$
in $L^2(\tfrac14\delta_0a\B)$ we also deduce that 
$$\|u_1 - u_2\|_{L^2(\tfrac14\delta_0a\B)} \neq0.$$

Next we intend to get a contradiction by showing that $u$ vanishes in a set of positive Lebesgue measure in $\B\subset\tfrac14\delta_0a\B$, which is impossible because the zero set of any harmonic function is a real analytic variety.

Recall that, by \rf{eq300a}, there exists a set $F_k\subset B(\xi,r_k)\setminus(\Omega_1\cup\Omega_2)$
such that
$\HH^{n+1}(F_k)\gtrsim r_k^{n+1}$.
Hence, denoting $G_k=T_k(F_k)$, we infer that 
$$\int_{G_k} |u_1^k - u_2^k|\,dx=0.$$
We may assume that $\chi_{G_k}$ converges weakly in $L^2(\B)$ to some non-negative function
$g\in L^2(\B)$.  Clearly, $\|g\|_{L^2(\B)}\lesssim1$ and
$$\int_\B g\, dx = \langle \chi_\B,g\rangle =
\lim_k \langle \chi_\B, \chi_{G_k}\rangle \gtrsim 1.$$
Also, by the strong convergence of $|u_1^k - u_2^k|$ and the 
weak convergence of $\chi_{G_k}$,
$$\int |u_1 - u_2|\,g\,dx = 
\lim_k \int |u_1^k - u_2^k|\,\chi_{G_k}\,dx =0,$$
which implies that $u_1 -u_2$ vanishes on a set of positive Lebesgue measure in $\B$ and provides the 
aforementioned contradiction.
\vv

It remains now to prove that $\omega_1^\infty =\omega_2^\infty$ in $\frac12a\B$. To this end, consider
a continuous  function $\vphi$ compactly supported in $\frac12\B$. Denote $\vphi_k=\vphi\circ T_k$.
For $i=1,2$ and all $k$ we write
\begin{align*}
\int \vphi \,d\omega_i^k & = \frac1{\omega_i^k(B(\xi,r_k))}\int \vphi_k\,d\omega_i \\
& =
\frac1{\omega_i^k(B(\xi,r_k))}\int_{E\cap B(\xi,ar_k)} \vphi_k\,d\omega_i +
\frac1{\omega_i^k(B(\xi,r_k))}\int_{B(\xi,ar_k)\setminus E} \vphi_k\,d\omega_i =: A_i^k+ B_i^k.
\end{align*}
Note that
$$|B_i^k| \leq \|\vphi_k\|_\infty\,\frac{\omega_i^k(B(\xi,ar_k)\setminus E)}{\omega_i^k(B(\xi,ar_k))}
\,\frac{\omega_i^k(B(\xi,ar_k))}{\omega_i^k(B(\xi,r_k))}\leq b\|\vphi_k\|_\infty\,\frac{\omega_i^k(B(\xi,ar_k)\setminus E)}{\omega_i^k(B(\xi,ar_k))}\to0,$$
as $k\to\infty$.
Hence to prove \rf{eq310} it suffices to show that
\begin{equation}\label{eqfac300}
\lim_k\frac1{\omega_1^k(B(\xi,r_k))}\int_{E\cap B(\xi,ar_k)} \vphi_k\,d\omega_1 = \lim_k\frac1{\omega_2^k(B(\xi,r_k))}\int_{E\cap B(\xi,ar_k)} \vphi_k\,d\omega_2.
\end{equation}  
To this end, we set
$$\frac1{\omega_2^k(B(\xi,r_k))}\int_{E\cap B(\xi,ar_k)} \vphi_k\,d\omega_2 = 
\frac{\omega_1^k(B(\xi,r_k))}{\omega_2^k(B(\xi,r_k))} \cdot\frac1{\omega_1^k(B(\xi,r_k))}\int_{E\cap B(\xi,ar_k)} \vphi_k\,h_{2,1}\,d\omega_1.$$
Since
$$\frac{\omega_1^k(B(\xi,r_k))}{\omega_2^k(B(\xi,r_k))} \to \frac1{h_{2,1}(\xi)},$$
we just have to check that
$$h_{2,1}(\xi)\,\lim_k\frac1{\omega_1^k(B(\xi,r_k))}\int_{E\cap B(\xi,ar_k)} \vphi_k\,d\omega_1 = \lim_k\frac1{\omega_1^k(B(\xi,r_k))}\int_{E\cap B(\xi,ar_k)} \vphi_k\,h_{2,1}\,d\omega_1.$$
This identity follows from the fact that $\xi$ is a Lebesgue point for $h_{2,1}$:
\begin{align*}
\biggl|\frac{h_{2,1}(\xi)}{\omega_1^k(B(\xi,r_k))} \int_{E\cap B(\xi,ar_k)} &\vphi_k\,d\omega_1 -\frac1{\omega_1^k(B(\xi,r_k))}\int_{E\cap B(\xi,ar_k)} \vphi_k\,h_{2,1}\,d\omega_1\biggr| \\
& \leq \frac 1{\omega_1^k(B(\xi,r_k))}\int_{B(\xi,ar_k)} |\vphi_k(x)|\,|h_{2,1}(\xi)-h_{2,1}(x)|\,d\omega_2(x)\\ & \leq \|\vphi\|_\infty \frac b{\omega_1^k(B(\xi,ar_k))}\int_{B(\xi,ar_k)} |h_{2,1}(\xi)-h_{2,1}(x)|\,d\omega_2(x) \to 0,
\end{align*}
as $k\to\infty$.
\fiproof
\vvv

\subsection{The case when $\Omega_i$ is not Wiener regular for some $i=1,2,3$}

Given any $\ve>0$, for each $i=1,2,3$ there exists a closed set $F_i$ such that ${\rm Cap}(F_i)<\ve$ and so that $\wt\Omega_i = \Omega_i\setminus F_i$
is Wiener regular (recall that $\rm Cap$ stands for the Newtonian capacity). For the detailed arguments the reader may consult Section 4 of \cite{HMMTV} (although we suspect that this result was known long before). Then, by the maximum principle (see Section 4 of \cite{HMMTV} for the precise justification), for any set $G\subset 
\partial\Omega_i \setminus F_i$, denoting by $\wt\omega_i$ the harmonic
measure of $\wt\Omega_i$ with respect to $p_i\in\Omega_i\cap\wt\Omega_i$,
$$\wt \omega_i(G)\leq \omega_i(G),$$
and, in particular, $$\wt \omega_i|_{\partial \Omega\setminus F_i}
\ll \omega_i|_{\partial\Omega\setminus F_i}.$$
One then easily deduces that, if $\ve$ is chosen small enough, then
there exists $\wt E\subset E$ such that $\wt \omega_i\approx \omega_i$
on $\wt E$ and $\wt\omega_i(\wt E)>0$ for each $i=1,2,3$ (we leave the details for the reader).
So we can apply the arguments above to the open sets $\wt\Omega_i$ and to
$\wt E$ to derive a contradiction.

\vvv
\subsection{The planar case $n=1$} 
We recall that in this case there are already purely analytic proofs by Bishop \cite{Bishop-arkiv} and Eremenko, Fuglede and Sodin \cite{EFS1}.
Because of this reason, we will only sketch the required changes
to adapt the proof of Subsection \ref{subdd**} to the planar case. 

The main reason why 
the arguments have to be modified is that Lemma \ref{l:w>G} does not hold for $n=1$, as far as we know. As shown in \cite[Subsection 4.4, close to (4.18)]{AHM3TV}, a reasonable substitute
of this lemma is provided by the following estimate:
\begin{align}\label{eq410}
|G(y,p) - G(z,p)| & \lesssim \frac{\omega^p(B(x_0,2\delta^{-1}r))}{\inf_{z\in B(x,2r)\cap \Omega} \omega^z(B(x,2\delta^{-1}r))} \\
&\quad + 
\int_{ B(x_0,3r)} \left(\left|\log\frac{r}{|y-\zeta|}\right| +  \left|\log\frac{r}{|z-\zeta|}\right|\right) \,d\omega^p(\zeta),\nonumber
\end{align}
valid for $x_0\in\partial\Omega$, $0<r<\diam(\partial \Omega)$,
 $y,z\in B(x_0,r)$, and  $p$ far away in $\Omega$.

Instead of defining the functions $u_i^k$ as in \rf{330}, we set
$$u_i^k(x) = \frac{r_k^{n-1}}{\omega_i(B(\xi,r_k))}\,\bigl(G_i(p_i,T_k^{-1}(x)) - G_i(p_i,z_i^k)\bigr),$$
where $z_i^k$ is some fix point in $B(\xi,r_k)\cap\Omega_i$.
Then, using \rf{eq410}, choosing appropriately $z_i^k$, and applying Fubini, one can check that
\begin{equation}\label{eq445}
\|u_i\|_{L^2(\delta_0 a \B)}\lesssim 1,
\end{equation}
which replaces \rf{eq331}, that is no longer valid.
On the other hand, as in \rf{eq301}, in this situation we also have
$$\int \vphi\,d\omega_i^k =\int \Delta\vphi\,u_i^k\,dy.$$

Because of Caccioppoli's inequality, \rf{eq445}, and the Rellich-Kondrachov theorem,
we still have that $u_i^k$ converges strongly in $L^2(\frac14\delta_0 a \B)$ to some function $u_i$ for $i=1,2$. Then we also set
$$u= u_1-u_2,$$
which turns out to be harmonic, by the same arguments as in the case $n\geq2$. Also, by analogous arguments to the ones of that case, one can check that $u$ is not identically constant in $\frac14\delta_0 a\B$, and then
one gets a contradiction by showing that $u$ is constant in a subset of positive Lebesgue measure, which is not possible because $u$ is harmonic.

\vvv

\end{document}